\documentclass[11pt]{article}

\usepackage[dvips]{epsfig} 
\usepackage{amssymb,amsmath,amsthm,mathrsfs} 
\usepackage{graphicx}
\usepackage{lineno}
\usepackage{multirow}

\usepackage[authoryear,sort&compress]{natbib}
\usepackage{url}
\usepackage{enumerate}
\usepackage{enumitem}
\usepackage{colordvi,color,pspicture}
\usepackage{psfrag}
\usepackage{pgf,tikz}
\usepackage{mathrsfs}
\usetikzlibrary{arrows}
\usetikzlibrary{decorations.markings}
\usepackage{authblk}

\newtheorem{theorem}{Theorem}[section]
\newtheorem{corollary}[theorem]{Corollary}
\newtheorem{lemma}[theorem]{Lemma}
\newtheorem{proposition}[theorem]{Proposition}
\newtheorem{observation}[theorem]{Observation}

\theoremstyle{definition}

\begin{document}
\title{On Total Domination and Minimum Maximal Matchings in Graphs}
\author{Selim Bahad{\i}r}
\affil{Department of Mathematics-Computer,\\
Ankara Y\i ld\i r\i m Beyaz\i t University, Turkey\\
sbahadir@ybu.edu.tr}
\date{\today}

\maketitle
\pagenumbering{roman} \setcounter{page}{1}

\pagenumbering{arabic} \setcounter{page}{1}

\begin{abstract}
\noindent
A subset $M$ of the edges of a graph $G$ is a matching if no two edges in $M$ are incident.
A maximal matching is a matching that is not contained in a larger matching.
A subset $S$ of vertices of a graph $G$ with no isolated vertices is a
total dominating set of $G$ if every vertex of $G$ is adjacent to at least one vertex in $S$.
Let $\mu^*(G)$ and $\gamma_t(G)$ be the minimum cardinalities of a maximal matching and a total dominating set in $G$, respectively.
Let $\delta(G)$ denote the minimum degree in graph $G$.
We observe that $\gamma_t(G)\leq 2\mu^*(G)$ when $1\leq \delta(G)\leq 2$ and $\gamma_t(G)\leq 2\mu^*(G)-\delta(G)+2$ when $\delta(G)\geq 3$.
We show that the upper bound for the total domination number is tight for every fixed $\delta(G)$.
We provide a constructive characterization of graphs $G$ satisfying $\gamma_t(G)= 2\mu^*(G)$ and a polynomial time procedure to determine whether $\gamma_t(G) = 2\mu^*(G)$ for a graph $G$ with minimum degree two.
\end{abstract}

\noindent
{\it Keywords: minimum maximal matching, total domination number} \\

\section{Introduction}
\label{sec:intro}
Let $G$ be a graph with vertex set $V(G)$ and edge set $E(G)$.
The neighborhood of a vertex $v\in V(G)$, denoted by $N(v)$, is the set of vertices adjacent to $v$.
The \emph{degree} of a vertex $v$ is the cardinality of $N(v)$ and denoted by $d(v)$.
The minimum degree of the graph $G$ is denoted by $\delta(G)$.
Throughout this paper, we only
consider simple, finite and undirected graphs without isolated vertices.

A set $S\subseteq V(G)$ of vertices is called a \emph{dominating set} of $G$ if every vertex of $V(G)\backslash S$ is adjacent to a member of $S$.
The \emph{domination number} $\gamma(G)$ is the minimum cardinality of a dominating set of $G$.
If $G$ has no isolated vertices, a subset $S\subseteq V(G)$ is called a \emph{total dominating set} of $G$ if every vertex of $V(G)$ is adjacent to a member of $S$.
In other words, $S$ is a total dominating set if $S$ is a dominating set and the subgraph of $G$ induced by $S$ has no isolated vertices.
The \emph{total domination number} of $G$ with no isolated vertices, denoted by $\gamma_t(G)$, is the minimum size of a total dominating set of $G$.

Two edges in $G$ are \emph{independent} if they share no vertex.
A set $M\subseteq E(G)$ consisting of pairwise independent edges is called a \emph{matching} in $G$.
A \emph{maximal matching} of $G$ is a matching that is not contained in a
larger matching in $G$.
The \emph{matching number} of $G$ is the maximum cardinality of a matching in $G$ and denoted by $\mu(G)$ (also $\alpha'(G)$ and $\nu(G)$).
Let $\mu^*(G)$ denote the minimum cardinality of a maximal matching in $G$.
A set $D\subseteq E(G)$ is called \emph{edge dominating set} if every edge not in $D$ is adjacent to at least one edge in $D$ and minimum cardinality of an edge dominating set is denoted by $\gamma'(G)$.
Any maximal matching is always an edge dominating set.
Furthermore, the size of a minimum edge dominating set equals the size of a minimum maximal matching, i.e., $\mu^*(G)=\gamma'(G)$ for every graph $G$ (see, \cite{yannakakis1980edge}).

Obtaining bounds on total domination number in terms of other graph parameters and classifying graphs whose total domination number attains an upper or lower bound are studied by many authors (see, Chapter 2 in \cite{yeo:2013}).
For example, \citeauthor{cockayne:1980} \cite{cockayne:1980} showed that if $G$ is a connected graph with order at least 3, then $\gamma_t(G)\leq 2|V(G)|/3$.
Moreover, \citeauthor{brigham:2000} \cite{brigham:2000} provided that a connected graph $G$ satisfies
$\gamma_t(G)= 2|V(G)|/3$ if and only if $G$ is a cycle of length 3 or 6, or $H\circ P_2$ for some connected graph $H$,
where $P_2$ is a path of length 2 and $H\circ P_2$ is obtained by identifying each vertex of $H$ by an end vertex of a copy of $P_2$.

It is well-known that $\gamma(G) \leq \gamma_t(G) \leq 2\gamma(G)$.
An open question in \cite{henning:2009} is to characterize the graphs $G$ with $\gamma_t(G)=2\gamma(G)$.
As partial answers of this problem, \citeauthor{henning:2001} \cite{henning:2001} provided a constructive characterization of trees,
\citeauthor{hou:2010} \cite{hou:2010} generalized it to block graphs and \citeauthor{sbdg:2018} \cite{sbdg:2018} presented a characterization of a large family of graphs (including chordal graphs) satisfying $\gamma_t(G)= 2\gamma(G)$.

For every graph $G$ with no isolated vertex it is true that $\gamma(G)\leq \mu(G)$.
However, the inequality $\gamma_t(G)\leq \mu(G)$ does not always hold.
\citeauthor{henning2008matching} \cite{henning2008matching} prove that $\gamma_t(G)\leq \mu(G)$ is valid for every claw-free graph $G$ with $\delta(G)\geq 3$ and
every $k$-regular graph $G$ with $k\geq 3$.
\citeauthor{henning2013total} \cite{henning2013total} show that
if all vertices in a connected graph $G$ with at least four vertices belong to a triangle, then $\gamma_t(G)\leq \mu(G)$.
Claw-free graphs with minimum degree three that have equal total domination and matching numbers are determined in \cite{henning2006total}, whereas every tree $T$ satisfying $\gamma_t(T)\leq \mu(T)$ is characterized in \cite{shiu2010some}.

Unlike the inequality $\gamma_t(G)\leq \mu(G)$,
the inequality $\gamma_t(G)\leq 2\mu^*(G)$ holds for every graph $G$ with no isolated vertex since the vertex set of a maximal matching is a total dominating set.
However, the inequality is tight only when the minimum degree is one or two.
We observe that if $\delta(G)\geq 3$, then $\gamma_t(G)\leq 2\mu(G)^*-\delta(G)+2$ and the equality is attained by infinitely many graphs for any given minimum degree.
In this paper,
we mainly study graphs satisfying the upper bound for total domination number,
$\gamma_t(G)= 2\mu^*(G)$, and refer to them
as $(\gamma_t,2\mu^*)$-graphs.
We characterize all $(\gamma_t,2\mu^*)$-graphs and present a process with polynomial time complexity to determine whether a given graph $G$ with $\delta(G)=2$ is a $(\gamma_t,2\mu^*)$-graph.

The rest of this paper is organized as follows:
Section \ref{sec:mainres} provides the main results and
the proofs of the main theorems are given in Section \ref{sec:proof}.
Discussion and conclusions are provided in Section \ref{sec:dis}.

\section{Main Results}
\label{sec:mainres}
\subsection{An Upper Bound for the Total Domination Number}
We first provide some definitions and notations required for the statement of the main results.
An edge joining vertices $u$ and $v$ is denoted as $uv$.
For a matching $M=\{u_1v_1,\dots, u_nv_n\}$ in a graph,
let $V(M)=\{u_1,v_1,\dots,u_n,v_n\}$
and $M_p(w)$ be the neighbor of $w$ in $M$ for every vertex $w$ in $V(M)$, i.e., $M_p(u_i)=v_i$ and $M_p(v_i) =u_i$ for $i=1,\dots,n$.
Notice that $M_p(M_p(w))=w$ for every $w\in V(M)$.
For a subset $A$ of $V(M)$ let $M_p(A)=\{M_p(v):v\in A\}$.
A set of vertices is called \emph{independent} if there is no edge joining two of the vertices in the set.
\begin{observation}\label{obs:maxmathind}
For every graph $G$ and every maximal matching $M$ in $G$, $V(G)\backslash V(M)$ is an independent set.
\end{observation}
\begin{observation}\label{obs:maxmatchtds}
Let $G$ be a graph with no isolated vertices and $M$ be a maximal matching of $G$.
Then, $V(M)$ is a total dominating set in $G$.
\end{observation}

\begin{proposition}\label{prop:gammatdeg3}
Let $G$ be a graph with no isolated vertices.
If $1\leq \delta(G) \leq 2$, then $\gamma_t(G)\leq 2\mu^*(G)$.
If $\delta(G)\geq 3$, then we have $\gamma_t(G)\leq 2\mu^*(G)-\delta(G)+2$.
\end{proposition}
\begin{proof}
Let $M$ be a minimum maximal matching in $G$ and $\delta=\delta(G)$.
By Observation \ref{obs:maxmatchtds} we get $\gamma_t \leq |V(M)|=2\mu^*(G)$ and hence, we obtain the inequality for $\delta \in \{1,2\}$.

When the minimum degree is at least three we can obtain a better inequality for the total domination number.
Suppose that $V(G)\backslash V(M)$ is not empty and let $x$ be vertex in $V(G)\backslash V(M)$.
Recall that $N(x)\subseteq V(M)$.
Let $A$ be a subset of $N(x)$ of size $\delta(G)-1$.
Consider the set $S=(V(M)\backslash M_p(A))\cup \{x\}$.
which is obtained by removing $\delta-1$ vertices from $V(M)$ and adding one vertex not in $V(M)$.
We claim that $S$ is a total dominating set of $G$.
For every $y$ not in $V(M)$ we have $N(y)\subseteq V(M)$ and $|N(y)|\geq \delta$.
Thus, $N(y)\cap S\neq \emptyset$.
Now let $v$ be a vertex in $V(M)$.
If $v\in A$, then $v$ is adjacent to $x\in S$.
If $v\notin A$, then $M_p(v)\in S$ and hence $M_p(v)\in N(v)\cap S$.
Consequently, we see that $S$ is a total dominating set and $\gamma_t \leq |S|=|M|-(\delta -1)+1=2\mu^*(G)-\delta +2$.

Next, assume that $V(G)\backslash V(M)$ is empty.
In this case, we see that $V(M)$ is the set of whole vertices in $G$.
As the minimum degree is $\delta$, any set obtained by removing $\delta -1$ vertices from $V(G)$ is a total dominating set and hence,
we obtain $\gamma_t(G)\leq |V(G)|-(\delta -1)=|V(M)|-\delta +1=2\mu^*(G)-\delta +1< 2\mu^*(G)-\delta +2$.
\end{proof}

\begin{figure}
\centering
\begin{tikzpicture}[line cap=round,line join=round,>=triangle 45,x=1.5cm,y=1.5cm]
\clip(0.5,2) rectangle (11.5,4.2);
\draw [line width=1pt] (2.5,3.3)-- (2,3.8);
\draw [line width=1pt] (2,3.8)-- (1.4,3.8);
\draw [line width=1pt] (1.4,3.8)-- (1,3.7);
\draw [line width=1pt] (2.5,3.3)-- (2.6,3.8);
\draw [line width=1pt] (2.6,3.8)-- (3,4);
\draw [line width=1pt] (3,4)-- (3.8,4);
\draw [line width=1pt] (2.5,3.3)-- (3.2,3.4);
\draw [line width=1pt] (3.2,3.4)-- (3.7,3.6);
\draw [line width=1pt] (3.7,3.6)-- (4.4,3.4);
\draw [line width=1pt] (2.5,3.3)-- (1.8,3.3);
\draw [line width=1pt] (1.8,3.3)-- (1.2,3.2);
\draw [line width=1pt] (1.2,3.2)-- (0.8,3.1);
\draw [line width=1pt] (6,4)-- (6,3);
\draw [line width=1pt] (6,3)-- (7,3);
\draw [line width=1pt] (7,3)-- (7,4);
\draw [line width=1pt] (7,4)-- (6,4);
\draw [line width=1pt] (7,4)-- (8,4);
\draw [line width=1pt] (8,4)-- (8,3);
\draw [line width=1pt] (8,3)-- (7,3);
\draw [line width=1pt] (8,4)-- (9,4);
\draw [line width=1pt] (9,4)-- (9,3);
\draw [line width=1pt] (9,3)-- (8,3);
\draw [line width=1pt] (10,4)-- (10,3);
\draw [line width=1pt] (10,3)-- (11,3);
\draw [line width=1pt] (11,3)-- (11,4);
\draw [line width=1pt] (11,4)-- (10,4);
\draw [line width=1pt,dotted] (9.22,3.48)-- (9.82,3.48);
\draw [line width=1pt] (9,4)-- (9.3,4);
\draw [line width=1pt] (9,3)-- (9.3,3);
\draw [line width=1pt] (10,3)-- (9.7,3);
\draw [line width=1pt] (10,4)-- (9.7,4);
\draw [shift={(2.6,3.5)},line width=1pt,dotted]  plot[domain=3.74:5.82,variable=\t]({1*0.8807979980267823*cos(\t r)+0*0.8807979980267823*sin(\t r)},{0*0.8807979980267823*cos(\t r)+1*0.8807979980267823*sin(\t r)});
\draw (2.3,2.5) node[anchor=north west] {$G_1$};
\draw (8.2,2.5) node[anchor=north west] {$G_2$};
\begin{scriptsize}
\draw [fill=black] (2.5,3.3) circle (2.5pt);
\draw [fill=black] (2,3.8) circle (2.5pt);
\draw [fill=black] (1.4,3.8) circle (2.5pt);
\draw [fill=black] (1,3.7) circle (2.5pt);
\draw [fill=black] (2.6,3.8) circle (2.5pt);
\draw [fill=black] (3,4) circle (2.5pt);
\draw [fill=black] (3.8,4) circle (2.5pt);
\draw [fill=black] (3.2,3.4) circle (2.5pt);
\draw [fill=black] (3.7,3.6) circle (2.5pt);
\draw [fill=black] (4.4,3.4) circle (2.5pt);
\draw [fill=black] (1.8,3.3) circle (2.5pt);
\draw [fill=black] (1.3,3.2) circle (2.5pt);
\draw [fill=black] (0.8,3.1) circle (2.5pt);
\draw [fill=black] (6,4) circle (2.5pt);
\draw [fill=black] (6,3) circle (2.5pt);
\draw [fill=black] (7,3) circle (2.5pt);
\draw [fill=black] (7,4) circle (2.5pt);
\draw [fill=black] (8,4) circle (2.5pt);
\draw [fill=black] (8,3) circle (2.5pt);
\draw [fill=black] (9,4) circle (2.5pt);
\draw [fill=black] (9,3) circle (2.5pt);
\draw [fill=black] (10,4) circle (2.5pt);
\draw [fill=black] (10,3) circle (2.5pt);
\draw [fill=black] (11,3) circle (2.5pt);
\draw [fill=black] (11,4) circle (2.5pt);
\draw [fill=black] (6.5,4) circle (2.5pt);
\draw [fill=black] (6.5,3) circle (2.5pt);
\draw [fill=black] (7.5,3) circle (2.5pt);
\draw [fill=black] (7.5,4) circle (2.5pt);
\draw [fill=black] (8.5,4) circle (2.5pt);
\draw [fill=black] (8.5,3) circle (2.5pt);
\draw [fill=black] (10.5,4) circle (2.5pt);
\draw [fill=black] (10.5,3) circle (2.5pt);
\end{scriptsize}
\end{tikzpicture}
\caption{Examples of $(\gamma_t,2\mu^*)$-graphs.
On the left, graph $G_1$ consists of $n$ paths of length three sharing a common end vertex.
Note that $\gamma_t(G_1)=2n$, $\mu^*(G_1)=n$ and $\delta(G_1)=1$.
On the right, graph $G_2$ is obtained by subdividing every horizontal edge of a $1\times n$ grid graph.
Notice that $\gamma_t(G_2)=2n+2$, $\mu^*(G_2)=n+1$ and $\delta(G_2)=2$.
}\label{fig:exgammat2mu}
\end{figure}
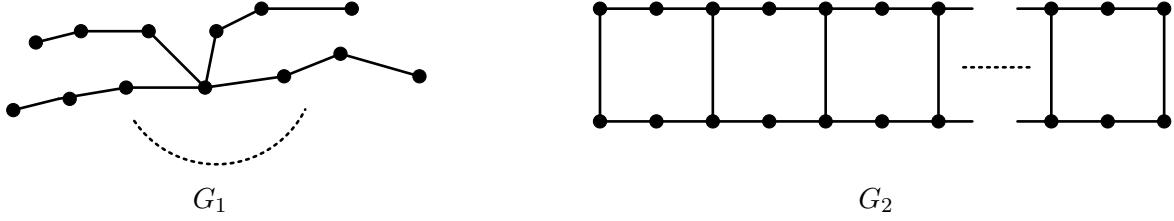
Proposition \ref{prop:gammatdeg3} implies that every  $(\gamma_t,2\mu^*)$-graph has minimum degree one or two.
Examples of $(\gamma_t,2\mu^*)$-graphs are illustrated in Figure \ref{fig:exgammat2mu}.
The following result shows that the inequality in Proposition \ref{prop:gammatdeg3} is tight when the minimum degree is at least three.

\begin{proposition}
For every positive integer $\delta \geq 3$, there exist infinitely many graphs $G$ with $\gamma_t(G)=2\mu^*(G)-\delta +2$ and $\delta(G)=\delta$.
\end{proposition}
\begin{proof}
Let $\delta \geq 3$ be a positive integer.
Let $M$ be disjoint union of $n$ copies of $K_2$ where $n\geq (\delta +1)/2$.
For every subset $A$ of $V(M)$ with $|A|=\delta$, add a new vertex $v_A$ whose neighborhood is $A$.
Let $G$ be the resulting graph.
The degree of a vertex in $V(M)$ is $\binom{2n-1}{\delta -1}\geq \binom{\delta}{\delta-1}=\delta$ and each vertex not in $V(M)$ is of degree $\delta$.
Therefore, the minimum degree of $G$ is $\delta$.
As $M$ is a maximal matching, we have $\mu^*(G)\leq n$.
Let $S$ be a total dominating set of $G$.
If $V(M)\subseteq S$, then $|S|\geq 2n> 2n-\delta +2$.
Otherwise, let $v$ be a vertex in $V(M)\backslash S$.
Since $S$ is a total dominating set, there exists a vertex $x\in S$ adjacent to $M_p(v)$.
Notice that $x\notin V(M)$.
If $|V(M)\backslash S|\geq \delta $, then there exists no vertex in $S$ adjacent to $v_A$ whenever $A\subseteq V(M)\backslash S$ and $|A|=\delta$ which contradicts with $S$ being a total dominating set.
Therefore, we get $|V(M)\backslash S|\leq \delta -1$ and then, $|S\cap V(M)|\geq 2n-(\delta-1)$.
Together with $x$, we see that $|S|\geq 2n-(\delta-1)+1=2n-\delta +2$.
Thus, in both cases we have the inequality $2n-\delta+2\leq |S|$ which gives $2n-\delta +2\leq \gamma_t(G)$.
Then we obtain the inequality chain
$2n-\delta+2\leq \gamma_t(G)\leq 2\mu^*(G)-\delta+2\leq 2n-\delta+2$ which implies $\gamma_t(G)= 2\mu^*(G)-\delta+2$.
\end{proof}

\subsection{Construction of $(\gamma_t,2\mu^*)$-Graphs}
In this paper, we not only show that there exist infinitely many $(\gamma_t,2\mu^*)$-graphs in both cases of the minimum degree but also provide a procedure which enables to construct all $(\gamma_t,2\mu^*)$-graphs.

A \emph{leaf} of a graph is a vertex with degree one,
while a \emph{support vertex} of a graph is a vertex adjacent to a leaf.
Let $sup(G)$ denote the set of all support vertices in the graph $G$.
Let $S^+(G)$ be the set of support vertices which are adjacent to a support vertex and let $S^-(G)=sup(G)\backslash S^+(G)$ which is the set of isolated vertices in the subgraph of $G$ induced by $sup(G)$.
A matching is called \emph{perfect} if it covers all the vertices in the graph.
\begin{theorem}\label{thm:MainThm1}
Let $G$ be a graph with $1\leq \delta(G)\leq 2$.
Then, $G$ is a $(\gamma_t,2\mu^*)$-graph if and only if there exists a maximal matching $M=M^+ \cup M^- \cup M^*$ in $G$ satisfying the following conditions:
\begin{enumerate}[label=(\roman*)]
	\item $M^+$ is a perfect matching of the subgraph of $G$ induced by $S^+(G)$.
	\item $S^-(G)\subseteq V(M^-)$ and every edge in $M^-$ joins a vertex from $S^-(G)$ and a vertex from $V(G)\backslash sup(G)$.
	\item For every $v\in S^-(G) \cup V(M^*)$, $M_p(v)$ is the unique neighbor of $v$ among the vertices in $V(M)$.
	\item For every distinct vertices $u$ and $v$ in $S^-(G) \cup V(M^*)$, whenever $u$ and $v$ have a common neighbor, there exists a vertex whose neighborhood is $\{M_p(u),M_p(v)\}$.
\end{enumerate}
\end{theorem}
The proof of Theorem \ref{thm:MainThm1} will be given in the next section.
Note that when $\delta(G)=2$ there is no leaf in $G$ and hence, $sup(G)=\emptyset$.
Therefore, Theorem \ref{thm:MainThm1} implies the following result for the graphs with minimum degree two.
\begin{corollary}\label{cor:maindeg2}
Let $G$ be a graph with $\delta(G)=2$.
Then, $G$ is a $(\gamma_t,2\mu^*)$-graph if and only if there exists a maximal matching $M$ in $G$ satisfying the following two conditions:
\begin{enumerate}[label=(\roman*)]
	\item For every $v\in V(M)$, $M_p(v)$ is the unique neighbor of $v$ among the vertices in $V(M)$.
	\item For every distinct vertices $u$ and $v$ in $V(M)$, whenever $u$ and $v$ have a common neighbor, there exists a vertex whose neighborhood is $\{M_p(u),M_p(v)\}$.
\end{enumerate}
\end{corollary}

By Theorem \ref{thm:MainThm1} we can characterize $(\gamma_t,2\mu^*)$-graphs in a constructive way.
Let $\mathcal{F}$ be the family of graphs obtained by following the steps below:
\begin{enumerate}
	\item
	Let $M$ be disjoint union of some copies of $K_2$ and $A$ be a set of vertices disjoint with $V(M)$.
	\item
	Mark some (might be none) of the vertices in $V(M)$.
	Let $L$ be the set of vertices $v$ in $V(M)$ such that $M_p(v)$ is unmarked.
	\item
	For each vertex $v$ in $A$ draw at least two edges joining $v$ to vertices in $V(M)$.
	\item
	For every leaf $v$ in $L$, join $v$ to some vertices in $A$.
	\item
	Draw some (might be none) edges joining two vertices of $V(M)\backslash L$.
	\item
	For every distinct vertices $u$ and $v$ in $L$,
	if $u$ or $v$ is marked and $N(u)\cap N(v)\neq \emptyset$,
	then add a new vertex whose neighborhood is $\{M_p(u),M_p(v)\}$
	unless such a vertex already exists.
	\item
	For every distinct vertices $u$ and $v$ in $L$,
	if none of $u$ and $v$ is marked and $(N(u)\cap N(v))\cup (N(M_p(u)) \cap N(M_p(v))\neq \emptyset$,
	then add two new vertices with neighborhoods $\{u,v\}$ and $\{M_p(u),M_p(v)\}$ unless such vertices already exist (when $v=M_p(u)$ add only one such vertex).
	\item
	Finally, for every marked vertex $v$, if $v$ is not a support vertex, then add some new vertices and join them to $v$.
\end{enumerate}
Construction of a graph in $\mathcal{F}$ is illustrated in Figure \ref{fig:exF}.
\begin{figure}\center
\begin{tikzpicture}[line cap=round,line join=round,>=triangle 45,x=1.5cm,y=1.5cm]
\clip(0.8,0.8) rectangle (11,6.2);
\draw [line width=1pt] (2,5)-- (3,5);
\draw [line width=1pt] (2,4)-- (3,4);
\draw [line width=1pt] (2,3)-- (3,3);
\draw [line width=1pt] (2,2)-- (3,2);
\draw (0.85,6) node[anchor=north west] {$A$};
\draw (2.3,6) node[anchor=north west] {$M$};
\draw (2.3,1.8) node[anchor=north west] {$L$};
\draw (8.3,1.3) node[anchor=north west] {$G$};
\draw [->,line width=1pt] (3.4,3.4) -- (3.8,3.4);
\draw [line width=1pt] (5,5)-- (6,5);
\draw [line width=1pt] (5,4)-- (6,4);
\draw [line width=1pt] (5,3)-- (6,3);
\draw [line width=1pt] (5,2)-- (6,2);
\draw [line width=1pt] (5,4)-- (4,5);
\draw [line width=1pt] (6,3)-- (4,4);
\draw [line width=1pt] (5,3)-- (4,4);
\draw [line width=1pt] (5,2)-- (4,5);
\draw [line width=1pt] (6,2)-- (4,3);
\draw [line width=1pt] (4,4)-- (5,5);
\draw [line width=1pt] (6,4)-- (6,5);
\draw [line width=1pt] (4,3)-- (5,3);
\draw [->,line width=1pt] (6.3,3.4) -- (6.7,3.4);
\draw [line width=1pt] (8,5)-- (9,5);
\draw [line width=1pt] (9,5)-- (9,4);
\draw [line width=1pt] (9,4)-- (8,4);
\draw [line width=1pt] (8,3)-- (9,3);
\draw [line width=1pt] (8,2)-- (9,2);
\draw [line width=1pt] (7,3)-- (8,3);
\draw [line width=1pt] (7,3)-- (9,2);
\draw [line width=1pt] (8,2)-- (7,5);
\draw [line width=1pt] (7,4)-- (8,5);
\draw [line width=1pt] (7,4)-- (9,3);
\draw [line width=1pt] (7,4)-- (8,3);
\draw [line width=1pt] (8,4)-- (7,5);
\draw [line width=1pt] (9,6)-- (8,4);
\draw [line width=1pt] (8,5)-- (8,6);
\draw [line width=1pt] (9,5)-- (10,5);
\draw [line width=1pt] (9,4)-- (10,3);
\draw [line width=1pt] (10,3)-- (9,2);
\draw [line width=1pt] (10,2)-- (9,3);
\draw [line width=1pt] (10,2)-- (8,3);
\draw [line width=1pt] (9,3)-- (10,1);
\draw [line width=1pt] (10,1)-- (8,2);
\draw [line width=1pt] (10,4)-- (9,5);
\draw [line width=1.2pt,dash pattern=on 1pt off 1pt on 1pt off 4pt] (1.8,4.2)-- (2.2,4.2);
\draw [line width=1.2pt,dash pattern=on 1pt off 1pt on 1pt off 4pt] (2.2,4.2)-- (2.2,3.2);
\draw [line width=1.2pt,dash pattern=on 1pt off 1pt on 1pt off 4pt] (2.2,3.2)-- (3.2,3.2);
\draw [line width=1.2pt,dash pattern=on 1pt off 1pt on 1pt off 4pt] (3.2,3.2)-- (3.2,1.8);
\draw [line width=1.2pt,dash pattern=on 1pt off 1pt on 1pt off 4pt] (3.2,1.8)-- (1.8,1.8);
\draw [line width=1.2pt,dash pattern=on 1pt off 1pt on 1pt off 4pt] (1.8,1.8)-- (1.8,4.2);
\begin{scriptsize}
\draw [fill=black] (1,5) circle (2.5pt);
\draw [fill=black] (1,4) circle (2.5pt);
\draw [fill=black] (1,3) circle (2.5pt);
\draw [fill=blue] (2,5) circle (2.5pt);
\draw [fill=blue] (3,5) circle (2.5pt);
\draw [fill=blue] (2,4) circle (2.5pt);
\draw [fill=black] (3,4) circle (2.5pt);
\draw [fill=black] (2,3) circle (2.5pt);
\draw [fill=black] (3,3) circle (2.5pt);
\draw [fill=black] (2,2) circle (2.5pt);
\draw [fill=black] (3,2) circle (2.5pt);
\draw [fill=black] (4,5) circle (2.5pt);
\draw [fill=black] (4,4) circle (2.5pt);
\draw [fill=black] (4,3) circle (2.5pt);
\draw [fill=blue] (5,5) circle (2.5pt);
\draw [fill=blue] (6,5) circle (2.5pt);
\draw [fill=blue] (5,4) circle (2.5pt);
\draw [fill=black] (6,4) circle (2.5pt);
\draw [fill=black] (5,3) circle (2.5pt);
\draw [fill=black] (6,3) circle (2.5pt);
\draw [fill=black] (5,2) circle (2.5pt);
\draw [fill=black] (6,2) circle (2.5pt);
\draw [fill=black] (7,5) circle (2.5pt);
\draw [fill=black] (7,4) circle (2.5pt);
\draw [fill=black] (7,3) circle (2.5pt);
\draw [fill=black] (8,5) circle (2.5pt);
\draw [fill=black] (9,5) circle (2.5pt);
\draw [fill=black] (8,4) circle (2.5pt);
\draw [fill=black] (9,4) circle (2.5pt);
\draw [fill=black] (8,3) circle (2.5pt);
\draw [fill=black] (9,3) circle (2.5pt);
\draw [fill=black] (8,2) circle (2.5pt);
\draw [fill=black] (9,2) circle (2.5pt);
\draw [fill=black] (9,6) circle (2.5pt);
\draw [fill=black] (8,6) circle (2.5pt);
\draw [fill=black] (10,5) circle (2.5pt);
\draw [fill=black] (10,3) circle (2.5pt);
\draw [fill=black] (10,2) circle (2.5pt);
\draw [fill=black] (10,1) circle (2.5pt);
\draw [fill=black] (10,4) circle (2.5pt);
\end{scriptsize}
\end{tikzpicture}
\caption{ Construction of a graph $G$ in $\mathcal{F}$. Marked vertices are shown in blue.
}
\label{fig:exF}
\end{figure}
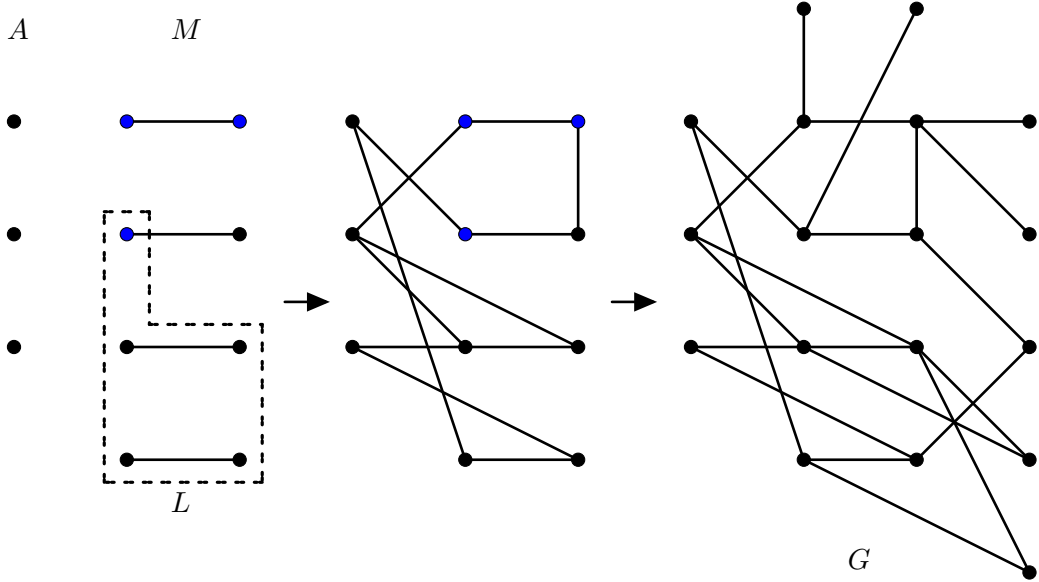	
In the procedure above, when at least one vertex is marked in the second step we obtain a graph with minimum degree one.
If no vertex is marked in the second step, then $L=V(M)$ and the resulting graph is of minimum degree two.

Let $G$ be a member of $\mathcal{F}$.
Observe that marked vertices corresponds to $sup(G)$.
Note also that $S^+(G)$ is the set of marked vertices $v$ such that $M_p(v)$ is also marked and $S^-(G)$ is the set of marked vertices $v$ such that $M_p(v)$ is not marked.
Therefore, we can split $M$ into three matchings $M^+=\{vM_p(v):v\in V(M), \text{ both } v \text{ and } M_p(v) \text{ are marked} \}$,
$M^-=\{vM_p(v):v\in V(M), v \text{ is marked but } M_p(v) \text{ is not  marked} \}$ and
$M^*=\{vM_p(v):v\in V(M),$ \text{ none of } $v \text{ and } M_p(v) \text{ is marked} \}$.
Thus, the set $L$ corresponds to $S^-(G)\cup V(M^*)$.
Under the assumption of condition (iv) of Theorem \ref{thm:MainThm1},
for every distinct vertices $u$ and $v$ in $V(M^*)$ we have $N(u)\cap N(v)\neq \emptyset$ if and only if $N(M_p(u))\cap N(M_p(v))\neq \emptyset $ since $M_p(M_p(x))=x$ for every $x$.
That is why two vertices are simultaneously inserted in the seventh step of the construction of $\mathcal{F}$.
Consequently, it is easy to verify that a graph $G$ has a maximal matching fulfilling the conditions in Theorem \ref{thm:MainThm1} if and only if $G$ belongs to $\mathcal{F}$ and hence, we obtain the following result.
\begin{corollary}\label{cor:MainThm2}
A graph $G$ is a $(\gamma_t,2\mu^*)$-graph if and only if $G\in \mathcal{F}$.
\end{corollary}
Corollary \ref{cor:MainThm2} enables us to construct all $(\gamma_t,2\mu^*)$-graphs and shows that there exist infinitely many $(\gamma_t,2\mu^*)$-graphs in both cases of the minimum degree.

\subsection{$(\gamma_t,2\mu^*)$-Graphs with Minimum Degree Two}
In this subsection, we provide a procedure to determine whether a given graph with minimum degree two is a $(\gamma_t,2\mu^*)$-graph.

Let $\mathcal{K}$ be the family of graphs consists of cycles $C_3$ with a common edge, that is, graphs isomorphic to a graph $G$ with $V(G)=\{u,v,w_1,\dots,w_n\}$ and $E(G)=\{uv,uw_1,vw_1,\dots,uw_n,vw_n\}$ for some positive integer $n$.
Let $d_2(G)$ be the set of vertices in $G$ of degree two, i.e., $d_2(G)=\{v\in V(G):d(v)=2\}$.

Let $G$ be a graph with connected components $G_1,\dots,G_n$ and no isolated vertex.
Since $\gamma_t(G)=\sum_{i=1}^n \gamma_t(G_i)$, $\mu^*(G)=\sum_{i=1}^n  \mu^*(G_i)$ and $\gamma_t(G_i)\leq 2\mu^*(G_i)$ for $i=1,\dots,n$,
we have $\gamma_t(G)=2\mu^*(G)$ if and only if $\gamma_t(G_i)=2\mu^*(G_i)$
for every $i\in\{1,\dots,n\}$.
Therefore, it is sufficient to consider only connected graphs.

For every connected graph $G\notin \mathcal{K}\cup \{C_6\}$,
let $\mathcal{M}=\mathcal{M}(G)$ be the set of edges $uv$ such that there exist $x\in N(u)\cap d_2(G)$ and $y\in N(v)\cap d_2(G)$ so that the edges $xu,uv,vy$ belong to the same induced $C_6$ in $G$.
In other words,
$\mathcal{M}$ can be constructed as follows:
Initially set $\mathcal{M}=\emptyset$ and then,
for every distinct vertices $x$ and $y$ in $d_2(G)$, if the subgraph of $G$ induced by $N(x)\cup N(y)\cup \{x,y\}$ is a $C_6$, then add both of the edges of the cycle which are incident to neither $x$ nor $y$ to the set $\mathcal{M}$.

For a given graph $G$ with $\delta(G)=2$, the following result enables to check whether $\gamma_t(G)=2\mu^*(G)$.
\begin{theorem}\label{thm:MainDeg2}
Let $G$ be a connected graph with $\delta(G)=2$.
Then, $G$ is a $(\gamma_t,2\mu^*)$-graph if and only if
$G\in \mathcal{K}\cup \{C_6\}$ or
$\mathcal{M}$ is a maximal matching satisfying the following two conditions:
\begin{enumerate}[label=(\roman*)]
	\item For every $v\in V(\mathcal{M})$, $\mathcal{M}_p(v)$ is the unique neighbor of $v$ among the vertices of $V(\mathcal{M})$.
	\item For every distinct vertices $u$ and $v$ in $V(\mathcal{M})$, whenever $u$ and $v$ share a neighbor, there exists a vertex with neighborhood $\{\mathcal{M}_p(u),\mathcal{M}_p(v)\}$.
\end{enumerate}
\end{theorem}
For general graphs,
both of the problems of finding the total domination number and finding the minimum maximal matching number are NP-complete
(see, \cite{pfaff1983np} and \cite{yannakakis1980edge}, respectively.)
However,
constructing the set $\mathcal{M}$ and checking whether it is a maximal matching and satisfies the conditions (i) and (ii) in Theorem \ref{thm:MainDeg2} can be easily done by an algorithm with polynomial time complexity.
Therefore, the problem of determining whether $\gamma_t (G)=2\mu^*(G)$ for a graph $G$ with $\delta(G)= 2$ is polynomial time solvable.

The girth of a graph $G$, denoted by $g(G)$, is the length of a shortest cycle (if any) in $G$.
Acyclic graphs (forests) are considered to have infinite girth.
Let $G$ be a connected $(\gamma_t,2\mu^*)$-graph with $\delta(G)=2$.
If $G\in \mathcal{K}\cup \{C_6\}$, then the girth of $G$ is either 3 or 6.
Otherwise, $\mathcal{M}$ has at least two edges and thus, $G$ has an induced $C_6$.
Therefore, $G$ has an induced $C_3$ or $C_6$ and hence, Theorem \ref{thm:MainDeg2} implies the following result.
\begin{corollary}\label{cor:girth}
If $G$ is a $(\gamma_t,2\mu^*)$-graph with $\delta(G)= 2$, then	$g(G)\leq 6$.
\end{corollary}

\section{Proofs of the Main Results}
\label{sec:proof}
\subsection{Proof of Theorem \ref{thm:MainThm1}}
We first present some results which are simple to observe and useful for the proofs.
\begin{observation}\label{obs:tdssup}
Let $G$ be a graph without isolated vertices.
Every total dominating set of $G$ contains $sup(G)$.
\end{observation}
\begin{observation}\label{obs:maxmatchsup}
For every graph $G$ and every maximal matching $M$ of $G$, we have $sup(G)\subseteq V(M)$.
\end{observation}

\begin{lemma}\label{lem:MT->}
Let $G$ be a $(\gamma_t,2\mu^*)$-graph.
Then $G$ has a maximal matching $M=M^+ \cup M^- \cup M^*$ satisfying the conditions in Theorem \ref{thm:MainThm1}.
\end{lemma}
\begin{proof}
Let $G$ be a graph such that $\gamma_t(G)=2\mu^*(G)$.
Let $\mu^*(G)=k$ and $M$ be a maximal matching of size $k$.
Note that $\gamma_t(G)=2k$ and $V(M)$ is a minimum total dominating set.

Let $v$ be a vertex in $V(M)$ such that $M_p(v)$ is not a support vertex.
We will show that $M_p(v)$ is the unique neighbor of $v$ in $V(M)$.
Suppose that $v$ has a neighbor in $V(M)$ other than $M_p(v)$.
Since $V(M)$ is a minimum total dominating set there must exist a vertex $x$ such that $N(x)\cap V(M)=\{M_p(v)\}$.
Clearly $x$ is not equal to $v$ or any other vertex in $V(M)$, that is,
$x\notin V(M)$.
But then, we see that $N(x)\subseteq V(M)$ by Observation \ref{obs:maxmathind} and hence, $x$ is a leaf adjacent to $M_p(v)$.
However, $M_p(v)$ is not a support vertex, contradiction.
Consequently, for every $v\in V(M)$ satisfying $M_p(v)\notin sup(G)$ we have $N(v)\cap V(M)=\{M_p(v)\}$.

Now, let $v$ be a support vertex in $S^+(G)$.
By Observation \ref{obs:maxmatchsup} we have $v\in V(M)$ and $v$ is adjacent to a support vertex which is also in $V(M)$.
Therefore, by the argument above we obtain that $M_p(v)$ is a support vertex and hence, by definition of $S^+(G)$ we see that $M_p(v)\in S^+(G)$.
Consequently, we obtain that $v\in S^+(G)$ implies $M_p(v) \in S^+(G)$ and
$M^+=\{vM_p(v):v\in S^+(G) \}$ is a perfect matching of the subgraph of $G$ induced by $S^+(G)$ and contained in $M$.

Let $M^-=\{vM_p(v):v\in S^-(G)\}$ and note that $M^-$ is a subset of $M$.
Clearly, by definition of $S^-(G)$ we have $v\in S^-(G)$ implies $M_p(v)\notin sup(G)$.
Therefore, $M^+$ and $M^-$ satisfy the first two conditions.

Now let $M^*=M\backslash (M^+\cup M^-)$ which is the set of edges in $M$ whose neither of endpoints is a support vertex.
Let $v$ be a vertex in $S^-(G)\cup V(M^*)$.
Recall that $M_p(v)$ is not a support vertex and therefore,
the third condition is also satisfied.

Finally, let $u$ and $v$ be distinct vertices of $S^-(G)\cup V(M^*)$ sharing a common neighbor.
Let $w\in N(u)\cap N(v)$ and note that $w$ is not in $V(M)$.
Consider the set
$A=(V(M)\backslash \{M_p(u),M_p(v)\}) \cup \{w\}$.
As $\gamma_t(g)=2k$ and $|A|=2k-1$, there exists a vertex $x$ adjacent to no vertex in $A$.
Clearly $x$ is not a vertex in $V(M)$ and therefore,we get  $N(x)\subseteq V(M)$.
Since $V(M)$ is a total dominating set, we have $N(x)\cap V(M)\neq \emptyset$ which implies $N(x)\subseteq \{M_p(u),M_p(v)\}$.
As none of $M_p(u)$ and $M_p(v)$ is a support vertex, $x$ is not a leaf and consequently we obtain $N(x)=\{M_p(u),M_p(v)\}$ and thus,
last condition is also satisfied.
\end{proof}

\begin{lemma}\label{lem:MT<-}
Let $G$ be a graph with no isolated vertex.
If $G$ has a maximal matching $M=M^+ \cup M^- \cup M^*$ satisfying the conditions in Theorem \ref{thm:MainThm1},
then $G$ is a $(\gamma_t,2\mu^*)$-graph.
\end{lemma}
\begin{proof}
Let $M=M^+ \cup M^- \cup M^*$ be a maximal matching fulfilling the conditions and $|M|=2k$.
Notice that it suffices to prove that $2k\leq \gamma_t(G)$,
since $\gamma_t(G)\leq 2\mu^*(G)\leq |M|=2k$.

Let $T$ be a total dominating set of $G$.
We will provide a one to one function $f:V(M)\rightarrow T$.
Let $U$ be the set of vertices $u$ in $V(M)$ such that $M_p(u)\notin T$.
By Observation \ref{obs:tdssup} and the first two conditions we have $U\subseteq S^-(G)\cup V(M^*)$.

Since $T$ is a total dominating set, for every vertex $u\in U$ there exists a vertex in $T$ adjacent to $u$ and set one of them as $f(u)$.
Note that by the third condition we have $f(u)\notin V(M)$.
We claim that $f(u)\neq f(v)$ for distinct vertices $u$ and $v$ in $U$.
Assume the contrary and let $u$ and $v$ be two distinct vertices of $U$ such that $f(u)=f(v)$.
Then we see that $u$ and $v$ have a common neighbor and hence, the fourth condition implies
the existence of a vertex $x$ adjacent to only $M_p(u)$ and $M_p(v)$.
On the other hand, by definition none of $M_p(u)$ and $M_p(v)$ is in $T$ and therefore,
there is no vertex in $T$ adjacent to $x$ which contradicts with $T$ being a total dominating set.

For every vertex $u\in V(M)\backslash U$ set $f(u)$ to be $M_p(u)$.
Then, it is easy to verify that $f(u)\neq f(v)$ for every distinct vertices $u$ and $v$ in $V(M)$, that is, $f:V(M)\rightarrow T$ is an injection and hence,
$|V(M)|=2k\leq |T|$ for any total dominating set $T$.
Consequently, we get $2k \leq \gamma_t(G)$ and thus, $\gamma_t(G)= 2k$.
\end{proof}
Combining the results of Lemmas \ref{lem:MT->} and \ref{lem:MT<-} gives Theorem \ref{thm:MainThm1}.

\subsection{Proof of Theorem \ref{thm:MainDeg2}}
It is clear that if $G\in \mathcal{K}\cup\{Cc_6\}$, then $\gamma_t(G)=2\mu^*(G)$ and $\delta(G)=2$.
By Corollary \ref{cor:maindeg2}, we see that if $\mathcal{M}$ is a maximal matching satisfying the conditions in Theorem \ref{thm:MainDeg2}, then $G$ is a $(\gamma_t,2\mu^*)$-graph.
Therefore, to prove Theorem \ref{thm:MainDeg2} the following result is sufficient by Corollary \ref{cor:maindeg2}.

\begin{lemma}
Let $G$ be a connected $(\gamma_t,2\mu^*)$-graph with $\delta(G)=2$ and $G\notin \mathcal{K}\cup \{C_6\}$.
Then $\mathcal{M}$ is the unique maximal matching satisfying the conditions (i) and (ii) in Corollary \ref{cor:maindeg2}.
\end{lemma}
\begin{proof}
Let $M$ be a maximal matching in $G$ satisfying the conditions (i) and (ii) in Corollary \ref{cor:maindeg2}.

We first show that $M\subseteq \mathcal{M}$.
Let $uv$ be an edge in $M$.
Then, since $G$ is connected and $G\notin \mathcal{K}$,
there exists a vertex $z\in (N(u)\cup N(v))\backslash \{u,v\}$ such that $N(x)\neq \{u,v\}$.
As the degree of $z$ is at least two,
$z$ has a neighbor $w$ which does not belong to $\{u,v\}$.
Note that $z$ is adjacent to at least one of $u$ and $v$.
Without loss of generality, suppose that $v$ is a neighbor of $z$.
Note also that $z\notin V(M)$ and $w\in V(M)$ by condition (i) and Observation \ref{obs:maxmathind}.
Therefore, $v$ and $w$ are distinct vertices in $V(M)$ both adjacent to $z$ and hence, there exist vertices $x$ and $y$ such that
$N(x)=\{M_p(v),M_p(w)\}=\{u,M_p(w)\}$ and $N(y)=\{v,w\}$.
Then, the subgraph of $G$ induced by $\{x,y,u,v,w,M_p(w)\}$ is a $C_6$ containing the edges $xu,uv,vy$ and thus, we obtain $uv\in \mathcal{M}$ which yields $M\subseteq \mathcal{M}$.

We next prove that $\mathcal{M}\subseteq M$ by contradiction.
Suppose that an edge $uv$ belongs to $\mathcal{M}\backslash M$.
Since $uv\in \mathcal{M}$, there exist vertices $w,t$ and $x,y\in d_2(G)$ such that subgraph of $G$ induced by $\{x,u,v,y,w,t\}$ has the edge set $\{xu,uv,vy,yw,wt,tx\}$.
As $uv\notin M$, condition (i) and Observation \ref{obs:maxmathind} imply that exactly one of $u$ and $v$ belongs to $V(M)$.
Without loss of generality, let $u\in V(M)$.
Then, $v\notin V(M)$ and hence, $y\in V(M)$ by condition (i) and Observation \ref{obs:maxmathind}.
Therefore, $u$ and $y$ are distinct vertices of $V(M)$ and hence there exists a vertex $z$ whose neighborhood is $\{u,y\}$.
Since $y$ has degree two, $z$ is either $v$ or $w$.
However, $u$ and $w$ are not adjacent and thus, we get $z=v$ and $v\in d_2(G)$.
Since $y\in d_2(G)\cap V(M)$ and $v\notin V(M)$, we get $w\in V(M)$ and $M_p(y)=w$.
As both $w$ and $y$ are in $V(M)$, we see that $t\notin V(M)$ and hence, we also get $x\in V(M)$ and $M_p(u)=x$ by condition (i).
Thus, $x$ and $w$ are distinct vertices in $V(M)$ and therefore, there exists a vertex $z$ with $N(z)=\{x,w\}$.
As $x$ is of degree two and $u$ is not a neighbor of $w$, we obtain $z=t$ and $t\in d_2(G)$.
Now note that on this cycle $C_6$ the vertices $x,y,v,t$ are of degree two.
Since $G$ is connected and $G\neq C_6$,
at least one of $u$ and $w$ has a neighbor $s$ not on this cycle.
Without loss of generality, let $s$ be adjacent to $w$.
As $w\in V(M)$ and $s\neq y$, we have $s\notin V(M)$.
Moreover, since the degree of $s$ is at least two, condition (i) implies that
$s$ is adjacent to a vertex $r$ in $V(M)\backslash \{w\}$.
Then, $w$ and $r$ share a neighbor and thus, by condition (ii) we see that there exists a vertex $q$ such that $N(q)=\{M_p(w),M_p(r)\}=\{y,M_p(r)\}$ and hence, $y$ is adjacent to $q$.
As $y\in d_2(G)$, $q$ must be the vertex $v$ and therefore, we get
$M_p(r)=u$ which gives $r=x$.
But then, $x$ is adjacent to three distinct vertices $u,t$ and $s$ which contradicts with $x\in d_2(G)$.
Consequently, we obtain $\mathcal{M}\subseteq M$ and the result follows.
\end{proof}

\section{Discussion and Conclusions}
\label{sec:dis}
In this paper, we study graphs $G$ for which the total domination number $\gamma_t(G)$ attains its upper bound $2\mu^*(G)$, that is $(\gamma_t,2\mu^*)$-graphs.
We provide not only a constructive characterization of  $(\gamma_t,2\mu^*)$-graphs but also
a polynomial time procedure to determine whether a given graph $G$ with $\delta(G)= 2$ is a $(\gamma_t,2\mu^*)$-graph.

Producing a polynomial time algorithm to determine $(\gamma_t,2\mu^*)$-graphs with at least one leaf is a topic of ongoing research.
Another potential research direction is to find a necessary and sufficient condition for the graphs $G$ satisfying $\delta(G)=\delta$ and $\gamma_t(G)=2\mu^*(G)-\delta+2$ for some or all positive integers $\delta\geq 3$.

\section*{Acknowledgments}
This work is supported by the Scientific and Technological Research
Council of Turkey (TUBITAK) under grant no. 118E799.

\end{document}